\documentclass{amsart}

\date{\today}
\usepackage{bm} 
\usepackage{mathpazo}\usepackage{amsfonts}
\usepackage{amssymb}
\usepackage{yhmath}
\usepackage{accents}
\usepackage{hyperref}
\usepackage{xcolor,tikz}
\usetikzlibrary{arrows,chains,matrix,positioning,scopes}
\usepackage[comma,sort&compress]{natbib}
\hypersetup{ colorlinks=true, linkcolor=blue, citecolor=blue,
  filecolor=blue, urlcolor=blue}

\usepackage{fancyhdr,latexsym,amsmath,amsfonts,amssymb,amsbsy,amsthm,url,mathrsfs,bbm}
\usepackage{amscd}

\allowdisplaybreaks
\newtheorem{theorem}{Theorem}[section]
\newtheorem{lemma}{Lemma}[section]
\newtheorem{proposition}{Proposition}[section]

\theoremstyle{definition}
\newtheorem{definition}{Definition}[section]

\theoremstyle{remark}
\newtheorem{remark}{Remark}[section]

\numberwithin{equation}{section}
\numberwithin{equation}{section}

\DeclareMathOperator{\R}{\mathbb{R}}

\DeclareMathOperator{\N}{\mathbb{N}}

\DeclareMathOperator{\one}{\mathbbm{1}} 

\newcommand{\cM}{\mathcal{M}}
\newcommand{\E}[1]{\mathsf{E}_{\mathcal W}\left[#1 \right]}
\newcommand{\var}[1]{\mathsf{Var}_{\mathcal W}\left(#1 \right)}
\newcommand{\cov}[1]{\mathsf{Cov}_{\mathcal W}\left(#1 \right)}
\newcommand{\prob}[1]{\mathcal{W}\left(#1 \right)}
\newcommand{\abs}[1]{\lvert #1 \rvert} 
\newcommand{\norm}[1]{\left\lvert\! \left\lvert#1 \right\rvert\!\right\rvert}

\renewcommand{\O}[1]{\mathrm{O}\left(#1\right)} 
\renewcommand{\o}[1]{\mathrm{o}\left(#1\right)} 
\newcommand{\la}{\left<}
\newcommand{\ra}{\right>}

\DeclareMathOperator{\diam}{\mathrm{diam}}
\DeclareMathOperator{\dime}{\mathrm{dim}}

\newcommand{\ca}[2]{\mathcal{I}\left(h_{\mu^{#1}_{#2}}\right)} 

\newcommand{\f}{\frac}  

\newcommand{\eps}{\epsilon}

\newcommand{\blank}[1]{}

\begin{document}

\title{Thick points for a Gaussian Free Field in $4$ dimensions}
\author[A. Cipriani]{Alessandra Cipriani}
\address{Weierstra{\ss}-Institut, Mohrenstra{\ss}e 39, 10117, Berlin, Germany}
\email{Alessandra.Cipriani@wias-berlin.de}

\author[R. S. Hazra]{Rajat Subhra Hazra}
\address{Institut f\"ur Mathematik\\ Universit\"at Z\"urich\\ Winterthurerstrasse 190\\ 8057-Zurich, Switzerland}
\email{rajat.hazra@math.uzh.ch}

\let\thefootnote\relax
\let\thefootnote\relax
\begin{abstract}
This article is concerned with the study of fractal properties of  thick points for 4-dimensional Gaussian Free Field.
We adopt the definition of Gaussian Free Field on $\R^4$ introduced by \cite{CJ} viewed as an abstract Wiener space with underlying Hilbert space $H^2(\R^4)$. We can
prove that for $0\leq a \leq 4$, the Hausdorf\/f dimension of the set of $a$-high points is $4-a$. We also show that the thick points give full mass to the Liouville Quantum Gravity measure on $\mathbb R^4$.
\end{abstract}

\maketitle

\section{Introduction} 
Random measures defined by means of log-correlated Gaussian fields $X$ and that can be formally written as ``$ m(\mathrm d\omega)=e^{\gamma X(\omega)}\mathrm d\omega$'' arise in conformal field theory and in probability. When $X$ is an instance of the Gaussian Free Field (GFF) then such measures are referred to as Liouville measures. The interest around such objects comes from physics and in particular from the understanding and proving the KPZ relation, formulated by Knizhnik, Polyakov and Zamolodchikov (\cite{KPZ}), which gives the relation between volume exponents derived using the quantum metric induced by $ m(\mathrm d\omega)$ and the Euclidean metric. Several interesting papers have been written to show this relation, and we refer to \cite{DS10, RaoVar10, RhoVarRev} for details about these results.
To construct such measures one has to rely on an approximation (cut-off) of the field and there are various methods to construct this approximation. 
 While on the one hand a more geometric approach (which explicitly relies on the structure of the field) is present in the work \cite{DS10}, the perspective of \cite{RobVar08, RaoVar10, RhoVarRev} dates back to the definition of \cite{Man72, Kah85} of multiplicative chaos, which deals with properties of the covariance kernel. These works extended the concept of multiplicative chaos of Kahane to a more general class of covariance kernels.

In this paper we focus our attention on the {\em multifractal formalism} of the underpinned Gaussian field, or with an equivalent terminology on its so-called {\em thick points}. To our knowledge the first rigorous study in this direction was made by Mandelbrot (in the collection \cite{BarMan04}) in the context of one-dimensional log-correlated Gaussian fields. In an interesting work,
\cite{HMP} showed that the Hausdorf\/f dimension of the set of $a$-thick points is $2-a$ for $0\le a\le 2$ for the planar GFF, and in \cite{RhoVarRev}, \cite{Kah85} such a result is shown for ``nice'' covariance kernels leading to multiplicative chaos. 
The set of thick points is relevant in the understanding the support of two dimensional Liouville quantum gravity (LQG). 
It was shown in fact in~\cite{DS10} that the LQG measure is almost surely supported on the thick points, in analogy to Kahane's similar results (\cite{Kah85}) on 1D Gaussian multiplicative chaos and to \cite[Theorem 4.1]{RhoVarRev} in higher dimensions. 
Our work being motivated by the definition of sphere average introduced by \cite{CJ} in dimension 4,  we prefer to stick to the more geometrical construction of the Gaussian free field rather than handling it as an instance of multiplicative chaos, although both approaches prove to be fruitful to investigate high points. Other than Chen and Jakobson's recent article and the developement of multiplicative chaos, the main motivation for considering such model comes from its discrete analogue which turns out to be related to the membrane model (cf. \cite{Kurt_d5}) defined on $\mathbb Z^d$. It is known that in dimension $4$ the model undergoes a phase transition in terms of the behavior
of the infinite Gibbs volume measure, as was proved in \cite{Kurt_d4}. Recently, some work on the fractal dimension of the thick points in this discrete setting has been carried through by \cite{Daviaud} for the 2D
discrete Gaussian Free Field and \cite{Cip13} on the discrete 4D membrane model.

 In this article we adhere to the sphere average process of \cite{CJ} and prove in Theorem~\ref{theo:support} that the set of thick points gives full mass to the LQG measure. In particular, we show in  Theorem~\ref{main theorem} that the set of $a$-thick points has Hausdorff dimension $4-a$ when $0\le a\le 4$. 
When $a>4$, the set of thick points is almost surely empty. The outline of the article is as follows. In Section~\ref{sec:model} we recall the model introduced by \cite{CJ} and state our main result more precisely. In Section~\ref{sec:support} we list some basic properties of the sphere average process and also provide a proof of~Theorem~\ref{theo:support} using a so-called \textit{rooted} or \textit{Peyri\`ere measure}.
The proof of Theorem~\ref{main theorem} is given in Sections~\ref{proof-upper} and~\ref{proof-lower} and relies on proving two different bounds. For the upper bound we use the version of the Kolmogorov-Centsov
theorem derived by \cite{HMP}. For the lower bound we use a standard finite-energy method and the Markov property of the GFF.

\section{GFF model and statement of the main results}\label{sec:model}
To keep the paper self contained we review in this section some definitions of the GFF on $\mathbb R^4$ from \cite{CJ} and state some properties of the sphere average process which will
be useful in deriving our main result.
In order to do so we begin with the definition of abstract Wiener space.
\begin{definition}[Abstract Wiener space, \cite{Str10}] \label{awsdef}
An \emph{abstract Wiener space} is a triple $\left(\Theta,H,\mathcal{W}\right)$, where
\begin{itemize}
\item $\Theta$ is a separable Banach space,
\item $H$ is a Hilbert space which is continuously embedded as a dense subspace of $\Theta$, equipped with the scalar product $(\cdot,\,\cdot)_H$,
\item $\mathcal W$ is a  Gaussian probability measure on $\Theta$ defined as follows.
\end{itemize}

Let $\Theta^*$ be the dual space of $\Theta$. Given any $x^*\in \Theta^*$ there exists a unique $h_{x^*}\in H$ such that for all $h\in H$,
$\la h,x^*\ra=(h,h_{x^*})_H$ where $\la \cdot,\,x^*\ra$ denotes the action of $x^*$ on $\Theta$. The sigma algebra $\mathcal B(\Theta)$  on $\Theta$ is
such that all the maps $\theta\mapsto \la\theta, \,x^* \ra$ are measurable.  $\mathcal W$ is a probability measure such that for all $x^*\in \Theta^*$,

\begin{equation}\label{eq:aws}
\mathsf E_{\mathcal W}\left[\exp\left(i\la \cdot, x^*\ra\right)\right]=\exp\left(-\f{\norm{h_{x^*}}^2_H}{2}\right) \, .
\end{equation}\end{definition}
Although the introduction of the set $\Theta$ is evidently important for the definition of the GFF,
its choice is not unique as explained in \cite{Str10}, Corollary 8.3.2 and afterwards. Moreover $\mathcal W(H)=0$ as $H$ is dense in $\Theta$. In our setting, we consider the underlying Hilbert space to be $H:=H^{2}\left(\mathbb{R}^{4}\right)$ which is the completion of the
Schwartz space $\mathcal S\left(\mathbb R^4\right)$ equipped with the inner product
$$( f_{1},f_{2})_{H}=\int_{\mathbb{R}^{4}}\left(I-\Delta\right)^{2}f_{1}\left(x\right)f_{2}\left(x\right)\mathrm d x\;\mbox{ for all }f_{1},f_{2}\in\mathcal{S}\left(\mathbb{R}^{4}\right).
$$
$H^{-2}\left(\mathbb{R}^{4}\right)$ is the
Hilbert space consisting of tempered distributions $\mu$ such that
\[
\left\Vert \mu\right\Vert _{H^{-2}}^{2}=\frac{1}{\left(2\pi\right)^{4}}\int_{\mathbb{R}^{4}}\left(1+\left|\xi\right|^{2}\right)^{-2}\left|\hat{\mu}\left(\xi\right)\right|^{2} \mathrm d \xi<\infty.
\]
where $\hat{\mu}$ is the Fourier transform. It is possible to identify $H$ with $H^{-2}$
through the linear isometry $\left(I-\Delta\right)^{-2}:\, H^{-2}\rightarrow H$. By abuse of notation we will call $h_{\nu}$ the image of $\nu\in H^{-2}$ under
$\left(I-\Delta\right)^{-2}$, that is, $h_{\nu}$ is the unique element
in $H$ such that $\la h,\nu\ra =( h,h_{\nu})_{H}$
for all $h\in H$. At this point we have to introduce another fundamental object for our work, the \emph{Paley-Wiener
integral} $\mathcal{I}\left(h_{\nu}\right)$. $\mathcal I$ is viewed as a mapping
\begin{eqnarray*}
\mathcal I:\,x^*\in \Theta^*&\mapsto& \mathcal I(h_{x^*})\in L^2(\mathcal W)\\
 && \theta\in \Theta \mapsto\mathcal [I(h_{x^*})](\theta):=\la \theta,\,x^*\ra .
\end{eqnarray*}
By \eqref{eq:aws}, we have $\left\{ \mathcal{I}\left(h_{\nu}\right):\nu\in H^{-2}\right\} $
is also a Gaussian family whose covariance is given by
\[
\mathsf E_{\mathcal{W}}\left[\mathcal{I}\left(h_{\nu_{1}}\right)\mathcal{I}\left(h_{\nu_{2}}\right)\right]=\la h_{\nu_{1}},h_{\nu_{2}}\ra_{H}=\la \nu_{1},\nu_{2}\ra_{H^{-2}}.
\]
Therefore $\mathcal I$ is an isometry from $\left\{h_{x^*}\,:\,x^*\in\Theta^*\right\}\to L^2(\mathcal W)$, and since the former set is dense in $H$,
 it admits a unique extension to the whole of $H$.
For every $x\in\mathbb{R}^{4}$ and $\epsilon>0$ denote as
$\sigma_{\epsilon}^{x}\in H^{-2}$ the tempered distribution given by
\[
\la f,\sigma_{\epsilon}^{x}\ra=\frac{1}{2\pi^{2}\epsilon^{3}}\int_{D(x,\eps)}f\left(y\right)\mathrm d \sigma\left(y\right),\,\mbox{ for all }f\in\mathcal{S}\left(\mathbb{R}^{4}\right),
\]
where $\mathrm d \sigma$ is the surface area measure
on $D(x,\eps)$, the sphere of radius $\eps$ around $x$. Interestingly, \cite{CJ} noted that $\left\{ \mathcal{I}\left(h_{\sigma_{\epsilon}^{x}}\right):\epsilon>0\right\} $ fails
to possess the Markov property and considered the following Gaussian family:
$$\left\{ \mathcal{I}\left(h_{\sigma_{\epsilon}^{x}}\right),\mathcal{I}\left(h_{\mathrm d \sigma_{\epsilon}^{x}}\right):x\in\mathbb{R}^{4},\epsilon>0\right\} ,$$ where
$\mathrm d \sigma_{\epsilon}^{x}$ the tempered distribution given by $\left\langle f,\mathrm d \sigma_{\epsilon}^{x}\right\rangle :=\frac{\mathrm d}{\mathrm d\epsilon}\left\langle f,\sigma_{\epsilon}^{x}\right\rangle $
for all $f\in\mathcal{S}\left(\mathbb{R}^{4}\right)$.  It is important to point out at this juncture that such a
collection is reminiscent of the double boundary conditions needed for the membrane model in the discrete case (\cite{Kurt_thesis}).
Let $\zeta:=(1,1)^T$ and $$ \mathbf{B}\left(r\right):=\left(\begin{array}{cc}
I_{1}\left(r\right)/r & I_{1}^{\prime}\left(r\right)\\
I_{2}\left(r\right)/r & I_{1}^{\prime\prime}\left(r\right)
\end{array}\right),$$
where  $I_{k}$ are the modified Bessel functions of order $k\in\mathbb{N}$.
Define
\begin{equation}
\mu_{\epsilon}^{x}:=\zeta^{\top}\mathbf{B}^{-1}\left(\epsilon\right)\left(\begin{array}{c}
\sigma_{\epsilon}^{x}\\
\mathrm d \sigma_{\epsilon}^{x}
\end{array}\right).
\end{equation}
It was shown in \cite{CJ} that $\mu_{\epsilon}^x\in H^{-2}(\mathbb R^4)$ and
$\left\{ \mathcal{I}\left(h{}_{\mu_{\epsilon}^{x}}\right):x\in\mathbb{R}^{4},\epsilon>0\right\} $ forms a Gaussian family with the correct Markovian properties and is the suitable candidate for the sphere average process.

\begin{definition}[Thick points of the sphere average]
For the sphere average process the set of {\it $a$-thick points} is defined as
\begin{equation}
T(a)=\left\{x\in \mathbb R^4:\lim_{\epsilon\to 0}\frac{\mathcal{I}\left(h{}_{\mu_{\epsilon}^{x}}\right)}{\sqrt{2 \pi^2} G(\epsilon)}=\sqrt{2a}\right\}.
\end{equation}
Here $G(\eps)= \var{\ca{x}{\eps}}$ and an explicit expression using Bessel functions is given in~\eqref{eq:green}.
\end{definition}
We would also need a definition of another set quite similar to the above:
\begin{equation}
T_{\geq}(a)=\left\{x\in \mathbb R^4:\limsup_{\epsilon\to 0}\frac{\mathcal{I}\left(h{}_{\mu_{\epsilon}^{x}}\right)}{\sqrt{2 \pi^2} G(\epsilon)}\geq \sqrt{2a}\right\}.
\end{equation}
It is easy to see that
$$T(a)\subset T_{\geq}(a).$$

One of the main results of~\cite{CJ} (Theorem 5) was to show the existence of the Liouville quantum gravity measure and the validity of the KPZ relation in $\mathbb R^4$.
Define a random measure on $\mathbb R^4$ by
$$m^\theta_\eps(\mathrm d x)= E^\theta_\eps(x) \mathrm d x,$$
 where 
$$E^\theta_\eps=\exp\left(\gamma\ca{x}{\eps}-\frac{\gamma^2}{2}G(\eps)\right).$$
If $\eps_n=\eps_0^n$ with $\eps_0\in(0,1)$ and $0<\gamma^2<2\pi^2$, then there exists a non-negative measure $m^\theta$ on $\mathbb R^4$ such that the following convergence holds for every $f\in C_c(\mathbb R^4)$:
\begin{equation}\label{eq:weak}
\int_{\mathbb R^4} f(x) m^\theta_{\eps_n}(\mathrm d x)\to \int_{\mathbb R^4} f(x) m^\theta(\mathrm d x) \text{ as } n \to\infty
\end{equation}
$\mathcal W$-almost surely and also in $L^2(\mathcal W)$. It is also known that this measure is almost surely positive.

 In the following Theorem we show that the set of thick points gives full measure to the LQG measure in $\mathbb R^4$.

\begin{theorem}\label{theo:support}
Let $0<\gamma^2<2\pi^2$, then for $a=\gamma^2/4\pi^2$ we have
$$m^\theta(T(a)^c) =0 \,  \, \, \mathcal W-a.s.$$
That is, the set $T(a)$ gives full mass to the measure $m^\theta(\cdot)$.
\end{theorem}

For the proof of Theorem~\ref{theo:support} we construct the \textit{rooted measure or Peyri\`ere measure}. For the use of rooted measures see~\cite{DS10, RhoVarRev}.

Before we state our main result on fractal properties of thick points,  we recall the definition of Hausdorff dimension and Hausdorff measure.\begin{definition}[Hausdorf\/f dimension]
 Let $X$ be a metric space and $S \subseteq X$. For every $d\geq0$ and $\delta>0$ define the Hausdorf\/f-$d$-measure in the following way:
$$
C^d_\delta(S)
:=\inf\Bigl\{\sum_i \diam(E_i)^d \,:\, E_1,\, E_2 ,\, E_3 ,\,\ldots,\text{ cover}\, S, \,\diam(E_i) \leq \delta \Bigr\},
$$
i.e. we are considering coverings of $S$ by sets of diameter no more than $\delta$.
Then
$$
C_{\mathcal H}^d (S) = \sup_{\delta>0} C^d_\delta(S) = \lim_{\delta\downarrow 0} C^d_\delta(S)
$$
is the Hausdorf\/f-$d$-measure of the set $S$.
The \emph{Hausdorff dimension} of $S$ is defined by
$$
  \dim_{\mathcal{H}}(S):=\inf\{d\ge 0:\, C_{\mathcal{H}}^d(S)=0\}.
$$
\end{definition}
\begin{theorem}\label{main theorem}
For $0\le a\le 4$, the Hausdorf\/f dimension of $T(a)$ is $4-a$. For $a>4$, we have that $T(a)$ is empty.
\end{theorem}

\begin{remark}
The above result shows similarity with the membrane model. In \cite{Cip13} it was shown that discrete fractal dimension of the $a$-high points is $4-4a^2$.

\end{remark}

To prove Theorem~\ref{main theorem} we apply some
of the techniques implemented in \cite{DPRZPlanar,DPRZSpatial} to show similar results for occupation measures of planar or spatial Brownian motion.

\section{GFF model and some estimates}\label{sec:support}
This section is devoted to providing some details about the behavior of the sphere average process, such as the covariance structure.
We then use them to derive a proof of Theorem~\ref{theo:support}.

\subsection{Some more properties of the sphere average process: covariance structure}
Let us denote as $D(0,R)$ the sphere centered at $0$ with radius $R>0$. Let $I_r, K_r$  be the modified Bessel functions of order $r\in\mathbb N\cup \{0\}$.
Define the positive function $G:\left(0,\infty\right)\mapsto \left(0,\infty\right)$
by
\begin{eqnarray}
G\left(r\right) & :=\left(-\frac{1}{4\pi^{2}}\right)\frac{2I_{1}\left(r\right)K_{1}\left(r\right)+2I_{2}\left(r\right)K_{0}\left(r\right)-1}{I_{1}^{2}\left(r\right)-I_{0}\left(r\right)I_{2}\left(r\right)}\,\, .
\label{eq:green}
\end{eqnarray}
It can be shown that $G$ is strictly decreasing and smooth, with $\lim_{r\to 0}G(r)=+\infty$ and $\lim_{r\to +\infty}G(r)=0$. 
It also follows from the properties of the Bessel functions that as $r$ decreases to $0$, $G(r)$ asymptotically behaves like $-\frac{1}{2\pi^2} \log r$.
Then, we have that
\begin{enumerate}
\item given $x\in\mathbb{R}^{4}$ and $\epsilon_{1}\geq\epsilon_{2}>0$\emph{,
\begin{equation}
\mathsf E_{\mathcal{W}}\left[\mathcal{I}\left(h_{\mu_{\epsilon_{1}}^{x}}\right)\mathcal{I}\left(h_{\mu_{\epsilon_{2}}^{x}}\right)\right]=
\mathsf E_{\mathcal{W}}\left[\mathcal{I}^{2}\left(h_{\mu_{\epsilon_{1}}^{x}}\right)\right]=G\left(\epsilon_{1}\right).\label{eq:cov concentric}
\end{equation}
}
\item Given $x,y\in\mathbb{R}^{4}$, $x\neq y$, and $\epsilon_{1},\epsilon_{2}>0$
with $\overline{D(x,\,\eps_1)}\cap\overline{D(y,\,\eps_2)}=\emptyset$,\emph{
\begin{equation}
\mathsf E_{\mathcal{W}}\left[\mathcal{I}\left(h_{\mu_{\epsilon_{1}}^{x}}\right)\mathcal{I}\left(h_{\mu_{\epsilon_{2}}^{y}}\right)\right]=\frac{1}{2\pi^{2}}K_{0}\left(\left|x-y\right|\right),\label{eq:cov nonoverlap}
\end{equation}
} where $K_0$ is the modified Bessel function of order $0$.
\item Given $x,y\in\mathbb{R}^{4}$, $x\neq y$, and $\epsilon_{1},\epsilon_{2}>0$ with $D(y,\,\eps_2)\subseteq D(x,\,\epsilon_1)$, \emph{
\begin{equation}
\begin{split}\mathsf E_{\mathcal{W}}\left[\mathcal{I}\left(h_{\mu_{\epsilon_{1}}^{x}}\right)\mathcal{I}\left(h_{\mu_{\epsilon_{2}}^{y}}\right)\right]=I_{0}\left(\left|x-y\right|\right)G\left(\epsilon_{1}\right)-\frac{1}{4\pi^{2}}\frac{I_{2}\left(\left|x-y\right|\right)}{I_{1}^{2}\left(\epsilon_{1}\right)-I_{0}\left(\epsilon_{1}\right)I_{2}\left(\epsilon_{1}\right)}.\end{split}
\end{equation}}
\end{enumerate}

The next lemma states one of the most useful and important properties of the spherical average process and is analogous to the properties of the two dimensional circular average process studied in \cite{DS10, HMP}. It shows that for fixed $x\in \mathbb R^4$,
the spherical average after a time change is a Brownian motion and in disjoint annuli two such motions are independent.
We briefly sketch the proof of the following lemma as it is an easy consequence after one compares the covariance structure.
\begin{lemma}\label{lem:time inversion}
\begin{enumerate}
 \item[(a)] Let $G(\cdot)$ be as in \eqref{eq:green} and for $x\in \mathbb R^4$, let $B(x,t)=\ca{x}{G^{-1}(t)}$. Then
$B(x,t)-B(x,t_1)$ has the same distribution as a standard Brownian motion for $t\ge t_1$.
\item[(b)] Given $x,y\in\mathbb R^4$ and $t_1\le t\le t_2$ and $s_1\le s\le s_2$ be such that $D(x, G^{-1}(s_1))\setminus D(x,G^{-1}(s_2))$ and
$D(y,G^{-1}(t_1))\setminus D(y,G^{-1}(t_2))$ are disjoint, then $\{B(x,s)-B(x,s_1)\}_{s_1\le s\le s_2}$ is independent of $\{B(y,t)-B(y,t_1)\}_{t_1\le t\le t_2}$.
\end{enumerate}
\end{lemma}
 \begin{proof}
(a)  It follows from \eqref{eq:cov concentric} that for $t_1\le s\le t$ one has
 \begin{eqnarray*}&&
  \mathrm{Cov}_{\mathcal W}(B(x,t)-B(x,t_1), B(x,s)-B(x,t_1))=\\
 &&=G(G^{-1}(s))-G(G^{-1}(t_1))-G(G^{-1}(t_1))+G(G^{-1}(t_1))=s-t_1
 \end{eqnarray*}
 Here we have used the fact that $G(\cdot)$ and $G^{-1}(\cdot)$ are decreasing functions and hence, as $t_1\le s\le t$ we have $G^{-1}(t_1)\ge G^{-1}(s)\ge G^{-1}(t)$.

(b)  As the annuli are disjoint it follows that
 $|x-y|>G^{-1}(t_1)+G^{-1}(s_1)\ge  G^{-1}(t)+G^{-1}(s)\ge G^{-1}(t_1)+G^{-1}(s_1)$ and hence again using \eqref{eq:cov nonoverlap} we obtain
 $$\mathrm{Cov}_{\mathcal W}\left(B(y,t)-B(y,t_1), B(x,s)-B(x,s_1)\right)= 0.$$

 \end{proof}
\subsection{Proof of Theorem~\ref{theo:support}}

Let $\Gamma$ be a compact subset of $\mathbb R^4$. Let $\mathcal B(\Gamma)$ be the Borel sigma algebra of subsets of $\Gamma$. We define a rooted measure on $\mathcal B(\Theta)\otimes \mathcal B(\Gamma)$ as
$$\cM(\mathrm d x \mathrm d \theta)=\frac{ m^\theta(\mathrm d x) \mathcal W(\mathrm d \theta)}{|\Gamma|}.$$
Here $|\Gamma|$ denotes the volume of the set $\Gamma$ with respect to the Lebesgue measure. Note that $\cM( \Theta\times \Gamma)= \E{m^\theta(\Gamma)}|\Gamma|^{-1}=1$ and as such $\cM$ is a probability measure on the space $\Gamma\times \Theta$.

Let  $r(t):=G^{-1}(t+G(R))$, $R>0$ fixed and define $$\widetilde B(x,t)(\theta):=\ca{x}{r(t)}(\theta)-\ca{x}{R}(\theta).$$
The following lemma allows us to view the random measure $m^\theta$ in a different way. We show that the joint distribution of $(x, \widetilde{B}(x,t))$ under $\cM(\mathrm d x \mathrm d \theta)$ is nothing but the distribution of 
$(x, \widetilde{B}(x,t)+\gamma t)$ under $\mathcal W(\mathrm d \theta) \mathrm d x$ and in the latter case the marginal on $\Theta$ does not depend on $x$.
\begin{lemma}\label{lem:transformation}
 Let $0<\gamma^2<2\pi^2$. For any compact set $\Gamma$ and any $F\in C_c(\mathbb R^4\times \mathbb R)$ we have
 \begin{equation}
 \int_\Theta \int_\Gamma F(x, \widetilde{B}(x,t)(\theta)) \cM(\mathrm d x \mathrm d \theta)=\frac1{|\Gamma|}\int_\Gamma \int_\Theta F(x, \widetilde{B}(x,t)(\theta)+\gamma t)\mathcal W(\mathrm d \theta) \mathrm d x.
 \end{equation}
\end{lemma}

\begin{proof}
Note that for almost every $\theta$, the map $x\in \Gamma\mapsto F(x, \widetilde{B}(x,t)(\theta))$ is continuous by Corollary 3 of \cite{CJ}. So from the weak convergence in~\eqref{eq:weak} we have that
$$\lim_{n\to\infty} \int_\Gamma F(x, \widetilde{B}(x,t))m^\theta_{\eps_n}(\mathrm d x)= \int_\Gamma F(x, \widetilde{B}(x,t))m^\theta(\mathrm d x).$$
Since the function in the integral is bounded we have for some constant $C$ and  for all $n$
$$\int_\Theta\int_\Gamma F(x, \widetilde{B}(x,t))m^\theta_{\eps_n}(\mathrm d x) \mathcal W(\mathrm d \theta)\le C|\Gamma|.$$
So by dominated convergence 
\begin{equation}\label{eqL1}
 \lim_{n\to\infty} \frac{1}{|\Gamma|}\int_\Theta\int_\Gamma F(x, \widetilde{B}(x,t))m^\theta_{\eps_n}(\mathrm d x) \mathcal W(\mathrm d \theta)=\int_\Theta\int_\Gamma F(x, \widetilde{B}(x,t) )\cM(\mathrm d x \mathrm d \theta).
\end{equation}
Note that for small enough $\epsilon>0$ 
$$ \mathrm{Cov}( \widetilde{B}(x,t), h_{\mu^x_{\eps}})= G(r(t))- G(R)= t$$
holds, so for $n$ large enough we have by Cameron-Martin theorem
\begin{align*}
\int_\Theta\int_\Gamma F(x, \widetilde{B}(x,t))m^\theta_{\eps_n}( \mathrm d x) \mathcal W(\mathrm d \theta)&=\int_\Theta\int_\Gamma F(x, \widetilde{B}(x,t))E^\theta_{\eps_n}(x)\mathrm d x \mathcal W (\mathrm d \theta)\\
&=\int_\Gamma \int_\Theta F(x, \widetilde{B}(x,t)(\theta)+\gamma t) \mathcal W(\mathrm d \theta)\mathrm d x.
\end{align*}
We have the required statement in the lemma using~\eqref{eqL1}.

\end{proof}
\begin{proof}[Proof of Theorem~\ref{theo:support}]
Using the fact $\mathrm{E}_{\mathcal W}\left[m^\theta(A)\right]=|A|$ for any bounded set $A$ it follows that the marginal of $\cM$ on $\Gamma$ is nothing but the normalized Lebesgue measure on $\Gamma$. Hence by Theorem 9.2.2.\ of \cite{Str10} there exists a Borel measurable map
$$x\in \Gamma \rightarrow \mathcal L_x(\cdot) \in M_1(\Theta),$$
where $M_1(\Theta)$ is the set of probability measures on $\Theta$ and the following holds
$$\cM( \mathrm d x \mathrm d \theta)= \mathcal L_x(\mathrm d \theta) \frac{\mathrm d x}{|\Gamma|}.$$

Note that $\mathcal L_x(\mathrm d \theta)$ is nothing but the regular conditional probability. Now using the above decomposition we have that
\begin{equation*}
 \int_\Theta \int_\Gamma F(x, \widetilde{B}(x,t)) \cM(\mathrm d x \mathrm d \theta)= \frac1{|\Gamma|}\int_{\Gamma}\int_\Theta F(x, \widetilde{B}(x,t))\mathcal{L}_x(\mathrm d \theta) \mathrm d x.
\end{equation*}
So from \eqref{lem:transformation} we have for any compact set $\Gamma$ and $F\in C_c(\mathbb R^4\times \mathbb R)$
\begin{equation}\label{eqL2}
\frac1{|\Gamma|}\int_{\Gamma}\int_\Theta F(x, \widetilde{B}(x,t))\mathcal{L}_x(\mathrm d \theta) \mathrm d x
=\frac1{|\Gamma|}\int_\Gamma \int_\Theta F(x, \widetilde{B}(x,t)+\gamma t)\mathcal W(\mathrm d \theta) \mathrm d x.
\end{equation}

If we denote  $\mu_x$ to be law of $\widetilde{B}(x,t)$ under $\mathcal{L}_x(\mathrm d \theta)$ and  $\nu$ be the law $\widetilde{B}(x,t)+\gamma t$ under $ \mathcal W (\mathrm d \theta)$ on $\mathbb R$ it is possible to see that $\nu$ is the law of a standard Brownian motion with a drift.
Since ~\eqref{eqL2} holds for any compact set $\Gamma$, it is easy to show that for almost every $x\in \mathbb R^4$,  $\mu_x=\nu$. If we take $a=\gamma^2/4\pi^2$ and use the fact that the sphere average process is a time inversion of a Brownian motion (see Lemma~\ref{lem:time inversion}), then the set of thick points can also be written as
$$T(a)=\left\{x\in\mathbb R^4: \lim_{t\to\infty}\frac{\widetilde{B}(x,t)}{t}=\gamma\right\}.$$
Now from the discussion above we have that

$$\cM( T(a)^c) = \frac1{|\Gamma|}\int_\Gamma \mathcal{L}_x(T(a)^c)\mathrm d x.$$
Since the law of $\widetilde{B}(x,t)$ under $\mathcal L_x$ is the same as the law of Brownian motion with a drift, the condition for the thick points gets satisfied with probability 1. So we have $\cM(T(a)^c)=0$,  which together with the fact that $m^\theta(\cdot)$ is a positive measure with probability $1$ proves the result.
\end{proof}

\section{Upper bound of Theorem~\ref{main theorem}}\label{proof-upper}
In this section we prove the upper bound. By the countable stability property, viz.
$$\displaystyle \hbox{dim}_{\mathcal H}\left( \bigcup_{i=1}^\infty E_i \right) = \sup_{1 \leq i \leq \infty} \hbox{dim}_{\mathcal H}( E_i )$$
it is enough to show that for $R\ge 1$
\begin{equation}
\hbox{dim}_{\mathcal H}{T_\ge(a,R)}=\hbox{dim}_{\mathcal H}\left\{x\in D(0,R):\limsup_{\epsilon\to 0}\frac{\mathcal{I}\left(h{}_{\mu_{\epsilon}^{x}}\right)}{\sqrt{2 \pi^2} G(\epsilon)}\geq\sqrt{2a}\right\}\le 4-a
\end{equation}
almost surely. Hence if we cover $\mathbb R^4$ with a countable union of balls of radius $R=1,\, 2,\,\ldots$, this will prove the upper bound.
The next proposition gives the local H\"older continuity of the process and through it we can determine a modification of the process
which has some uniform estimates on the increments.
It is similar to Proposition 2.1 of \cite{HMP} and uses Lemma C.1 of \cite{HMP}. The proof also uses some finer estimates on the covariance functions and some bounds on Bessel functions which are provided in the Appendix.

\begin{proposition}\label{prop:KC}
There exists a modification $\widetilde{X}$ of the process $\{\ca{z}{t}:\, z\in D(0,R),\,t\in(0,1)\}$ such that for every $0<\gamma<\frac12$ and $\eps, \zeta>0$ there exists $M>0$ such that the following holds:
\begin{equation}\label{eq:KC} |\widetilde{X}(z,r) - \widetilde{X}(w,s)| \leq M \left( \log \frac{1}{r} \right)^\zeta \frac{|(z,r) - (w,s)|^\gamma}{r^{(1+\eps)\gamma}},\end{equation}
for all $z,w \in D(0,R)$ and $r,s \in (0,1]$ with $1/2 \leq r/s \leq 2$.

\end{proposition}
\begin{proof}
Consider now $x,\,y \in D(0,R)$, $\epsilon_1,\,\epsilon_2\in(0,1)$ and we abbreviate
$$H_{\eps_1,\eps_2}(x,y):=\cov{\ca{x}{\eps_1},\,\ca{y}{\eps_2}}.$$
We distinguish between three cases:
\begin{description}
 \item[Case 1] Let $x=y$.  By Lemma \ref{lemma:bound_G}, we have
      \begin{align*}
   \abs{H_{\epsilon_1,\,\epsilon_1}(x,x)-H_{\epsilon_2,\,\epsilon_1}(x,x)}&\leq \abs{H_{\epsilon_1,\,\epsilon_1}(x,x)-H_{\epsilon_1,\,\epsilon_2}(x,x)}+\abs{H_{\epsilon_2,\,\epsilon_1}(x,x)-H_{\epsilon_1,\,\epsilon_2}(x,x)} \\
&\stackrel{\eqref{eq:cov concentric}}{\leq}\abs{G(\epsilon_1)-G{(\epsilon_1\vee \epsilon_2)}}+\abs{G(\epsilon_2)-G{(\epsilon_1\vee \epsilon_2)}}\\
&\leq C\frac{\abs{\epsilon_1-\epsilon_2}}{\epsilon_1\wedge \epsilon_2}.
      \end{align*}
Here we have used that $\abs{\log(x/y)}\leq \frac{\abs{x-y}}{x \wedge y}$.
\item[Case 2] Let $\overline{D(x,\,\eps_1)}\cap\overline{D(y,\,\eps_2)}=\emptyset$. In this case $\abs{x-y}>\epsilon_1+\epsilon_2>\epsilon_1$. Then
\begin{align*}
\abs{H_{\epsilon_1,\epsilon_1}(x,x)-H_{\epsilon_1,\epsilon_2}(x,y)}&=\abs{G(\epsilon_1)-\frac{1}{2 \pi^2}K_0(\abs{x-y})}\\
&\leq -C(\log \epsilon_1+\log(\abs{x-y}))\leq \frac{\abs{x-y}}{\epsilon_1}.
\end{align*}
Similarly one can show that $\abs{H_{\epsilon_2,\epsilon_2}(y,y)-H_{\epsilon_1,\epsilon_2}(x,y)}\leq  \frac{\abs{x-y}}{\epsilon_1}$.
\item[Case 3] Let $\overline{D(y,\,\epsilon_2)}\subseteq D(x,\,\epsilon_1)$.
\begin{align*}
\abs{H_{\epsilon_1,\,\epsilon_1}(x,x)-H_{\epsilon_1,\,\epsilon_2}(x,y)}&\leq \abs{G(\epsilon_1)(1-I_0(\abs{x-y}))}
+C\frac{I_2(\abs{x-y})}{I_1^2(\epsilon_1)-I_0(\epsilon_1)I_2(\epsilon_1)}\\
&\leq-C\log\epsilon_1\abs{x-y}^2+\frac{\abs{x-y}^2}{\epsilon_1^2}\leq C\frac{\abs{x-y}}{\epsilon_1}.
\end{align*}

\end{description}
Combining these three cases we obtain that
\begin{equation}
 \var{\ca{x}{\epsilon_1}-\ca{y}{\epsilon_2}}\leq C\frac{\abs{x-y}+\abs{\epsilon_1-\epsilon_2}}{\epsilon_1\wedge \epsilon_2}.
\end{equation}
Since $\ca{x}{\epsilon_1}-\ca{y}{\epsilon_2}$ is Gaussian,
$$
\E{\abs{\ca{x}{\epsilon_1}-\ca{y}{\epsilon_2}}^\alpha}\le C \left(\frac{\abs{x-y}+\abs{\epsilon_1-\epsilon_2}}{\epsilon_1\wedge \epsilon_2}\right)^{\alpha/2}.
$$
We can find $\alpha$ and $\beta$ large enough such that $\abs{\f{\beta}{\alpha}-\f{1}{2}}<\delta$, and consequently by \cite[Lemma C.1]{HMP} there exists a modification $\widetilde X(x,\epsilon)=\ca{x}{\epsilon}$
a.s.\ on $L^2(\mathcal W)$ satisfying \eqref{eq:KC}.
\end{proof}

In this section for the proof of the upper bound we work with this modification which we also denote by $\ca{x}{t}$. Recall that $B(x,t)= \ca{x}{G^{-1}(t)}$.

\begin{proof}[Proof of the upper bound]

Let $\varepsilon>0$ and $\gamma \in (0,1/2)$, $\zeta \in (0,1)$ and denote $\tilde \gamma:=(1+\varepsilon)\gamma$. Also let $K:=\varepsilon^{-1}$, $r_n:= n^{-K}$.

Define the set
$$U_R:=
\left\{x\in D(0,R):\limsup_{n\to +\infty}\frac{\mathcal{I}\left(h{}_{\mu_{r_n}^{x}}\right)}{\sqrt{2 \pi^2} G(r_n)}\geq\sqrt{2a}\right\}.
$$
We first show that
\begin{equation}\label{upper:subset}
 T_{\ge}(a, R)\subset U_R.
\end{equation}
For $x \in T_{\geq }(a,R)$  and for $t\in (G(r_n),\,G(r_{n+1}))$ we write $B(x,G(r_n))=B(x,G(r_n))-B(x,t)+B(x,t)
$ .  By Proposition~\ref{prop:KC} we have
\begin{align}
\abs{B(x,t)-B(x,G(r_n))}&\leq M\left(\log\left(\frac{1}{G^{-1}(t)}\right)\right)^\zeta\frac{\left(G^{-1}(t)-r_n\right)^\gamma}{G^{-1}(t)^{\tilde \gamma}}\nonumber\\
&\le M(\log(n+1))^\zeta \frac{\left(r_{n+1}-r_n\right)^\gamma}{r_{n+1}^{\tilde \gamma}}
=\O{  (\log n)^\zeta}.\label{eq:bound_O_log}
\end{align}

Hence using the fact that $G(r_n)\sim C \log n$ for $n\to +\infty$ and $\zeta<1$ we have
$$
\abs{\frac{B(x,G(r_n))-B(x,t)}{\sqrt{2 \pi^2} G(r_n)}}=\O{\frac{(\log n)^\zeta}{G(r_n)}}=\o{1}.
$$
Now \eqref{upper:subset} follows as we have
$$
\limsup_{n\to +\infty}\frac{B(x,G(r_n))}{\sqrt{2 \pi^2} G(r_n)}\geq \limsup_{t\to +\infty}\frac{B(x,t)}{\sqrt{2 \pi^2} t}\geq\sqrt{2 a}.
$$

The next step is to determine a cover for the set $U_R$. In view of that, let $(x_{nj})_{j=1}^{\bar k_n}$ be a maximal collection of points in $D(0,R)$ such that $\inf_{l\neq j}\abs{x_{nj}-x_{nl}}\ge r_n^{1+\varepsilon}$.
Denote
$$
\mathcal A_n:=\left\{j:\,\frac{\abs{B(x_{nj},G(r_n))}}{\sqrt{2 \pi^2} G(r_n)}\geq \sqrt{2a} -\delta(n)\right\}
$$
with $\delta(n)=C(\log n)^{\zeta-1}$ (the constant $C$ will be tuned later according to \eqref{eq:const_C}).
For any $x \in D(0,R)$, there exists $j\in\left\{1,\,\ldots,\,\overline k_n\right\}$ such that $x\in D\left(x_{nj},r_n^{1+\varepsilon}\right)$.
By \eqref{eq:bound_O_log} we have,
\begin{eqnarray}
&&\frac{\abs{B(x_{nj},G(r_n))-B(x,G(r_n))}}{\sqrt{2 \pi^2} G(r_n)}\leq C(\log n)^\zeta\frac{\abs{x-x_{nj}}^\gamma}{G(r_n)^{\tilde \gamma+1}}\nonumber\\
&&=\delta(n)\frac{\log n}{G(r_n)}\leq C\delta(n)\label{eq:const_C}
\end{eqnarray}
which implies, renaming possibly $\delta(n)$,
$$
\frac{B(x_{nj},G(r_n))}{G(r_n)}\geq \sqrt{2 a}-\delta(n).
$$
 Hence we have $j\in \mathcal A_n$. Therefore for all $N\geq 1$,  $\bigcup_{n\geq N}\bigcup_{j\in \mathcal A_N}D\left(x_{nj},r_n^{1+\varepsilon}\right)$ covers $U_R$ with sets having maximal diameter $2 r_n^{1+\varepsilon}$.
Next we claim that
\begin{equation}\label{upper:expectation}
\E{|\mathcal A_n|}\le C (\log n)r_n^{a-4(1+\varepsilon)+\mathrm{o}(1)}.
\end{equation}

Assume~\eqref{upper:expectation} for the moment. If we choose $\alpha:=4-a+\varepsilon\frac{4+a}{1+\varepsilon}$ we have
\begin{eqnarray*}
&&\E{\sum_{n\ge N}\sum_{j \in \mathcal A_n}\diam(D(x_{nj},r_n^{1+\varepsilon}))^\alpha}\le \sum_{n\ge N}(\log n) r_n^{(1+\varepsilon)\alpha+a-4(1+\varepsilon)+\o{1}}\\
&&\le \sum_{n\ge N}(\log n) r_n^{4\varepsilon+\o{1}}=C\sum_{n\ge N}(\log n) n^{-4+\o{1}}<+\infty.
\end{eqnarray*}
Therefore $\sum_{n\ge N}\sum_{j \in \mathcal A_n}\diam(D(x_{nj},r_n^{1+\varepsilon}))^\alpha<+\infty$ a.s.\ and this implies  $\dime_{\mathcal H}(T_\ge(a,r))\le 4-a$ a.s.\ by letting $\varepsilon \downarrow 0$.
This completes the proof of the upper bound provided we show~\eqref{upper:expectation}. We first estimate $\prob{j\in \mathcal A_n}$ as follows:
\begin{eqnarray*}
&&\prob{j\in \mathcal A_n}=\prob{\frac{\abs{B(x_{nj},G(r_n))}}{\sqrt{ G(r_n)}} \ge (\sqrt{2 a}-\delta(n) )\sqrt{2\pi^2}\sqrt{G(r_n)}}\\
&&\leq C (a+\o{1})G(r_n)\exp\left( -a\left(1-\o{1}\right)^2 2\pi^2G(r_n)  \right)\leq C(\log n) r_n^{a+\o{1}},
\end{eqnarray*}
since $G(r_n)\sim -\frac{ \log r_n}{2\pi^2}$ as $n\to+\infty$. Furthermore
\begin{eqnarray*}
&& \E{|\mathcal A_n|}\leq C(\log n) \overline k_n r_n^{(a+\o{1})}\leq (\log n) r_n^{a+\o{1}-4(1+\varepsilon)}.
\end{eqnarray*}
This proves~\eqref{upper:expectation} and hence the upper bound.

Now we show that for every $R>1$, $T_\ge (a,R)$  is empty for $a>4$ using the above estimates. Note that
\begin{eqnarray*}
&&\sum_{n\ge 1}\prob{\abs{\mathcal A_n}>1}\leq \sum_{n\ge 1}\E{\abs{\mathcal A_n}}\le \sum_{n\ge 1} r_n^{a-4(1+\varepsilon)}=\sum_{n\ge 1} r_n^4<+\infty
\end{eqnarray*}
and hence by the Borel-Cantelli lemma we can conclude that, if $\varepsilon$ becomes arbitrarily small, $\abs{\mathcal A_n}=0$ eventually and so $T_\ge(a,R)$ is empty for $a>4$ with probability one.
\end{proof}

\section{Lower bound of Theorem~\ref{main theorem}}\label{proof-lower}

To derive the lower bound we use the energy method. For detailed use of this method see Section 4.3 of \cite{MoerPer}.
The $\alpha$-th energy of a measure $\mu$ is given by
\[
 I_\alpha(\mu)=\iint \frac{\mathrm d\mu(x)\mathrm d\mu(y)}{|x-y|^\alpha}.
\]

Given a set $A$, if we can find a measure $\rho$ such that $I_\alpha(\rho)<\infty$ then $\dime_{H}(A)>\alpha$.
For this, partition the hypercube $J:=[0,1]^4$ into $s_n^{-4}$ smaller hypercubes of radius $s_n=\frac{1}{n!}$. Let $x_{ni}$ be the centers of these hypercubes and $C_n$ be the set of these centers.
Define $t_m:=G(s_m)$ for all $m\le n$. Note that since $G$ is decreasing we have that $t_m$ is increasing and also using the asymptotic expansion of $G$ we have, $t_m=-\f{\log s_m}{2\pi^2}(1+\o{1})$. Let $A_m(x)$, $B_m(x)$ be the events
$$
A_m(x):=\left\{\sup_{t_m<t\le t_{m+1}}\abs{B(x,t)-B(x,t_m)-\sqrt{4a\pi^2}(t-t_m)}\le \sqrt{t_{m+1}-t_m}\right\},
$$
$$
B_m(x):=\left\{\sup_{t\ge t_{m}}\abs{B(x,t)-B(x,t_m)}-t\le 1-t_m \right\}.
$$
We say that $x$ is an {\it $n$-perfect $a$-thick point} if $E^n(x):=\bigcap_{m\le n}A_m(x)\cap B_{n+1}(x)$ occurs. Note that $B_{n+1}(x)$ is independent of the other events.
We introduce a random variable $Y_{ni}$ for $i=1,\, \cdots,\, |C_n|$ such that
$$Y_{ni}=\begin{cases}
          1 &\text{ if $x_{ni}$ is an $n$-perfect $a$-thick point,}\\
         0  & \text{ otherwise. }
         \end{cases}$$

Fix $t_m<t\le t_{m+1}$ and on the event $E^n(x)$ we have, as $m\to\infty$,
\begin{eqnarray}
 &&\lvert B(x,t)-B(x,t_1)-\sqrt{4a\pi^2}(t-t_1)\rvert=  \o{m\log m}=\o{t}.\label{eq:o(t)}
\end{eqnarray}
Define now the set of perfect $a$-thick points as
$$
P(a):=\bigcap_{k\geq 1}\overline{\bigcup_{n\geq k}\bigcup_{z\in C_n(a)}S(z,s_n)},
$$
where $C_n(a)$ is the set of centers of which $x_{ni}$ is a $n$-perfect thick point and $S(z,r)$ is a hypercube of radius $r$ centered around $z$.
Let
$$ T(a,J):=\left\{x\in J: \lim_{t\to \infty}\frac{\ca{x}{G^{-1}(t)}}{\sqrt{2\pi^2} t}=a\right\}\subset T(a).$$
\begin{lemma}
\begin{equation}\label{lemma:inclusion_thick_pts}
P(a)\subseteq T(a,J).
\end{equation}
\end{lemma}

%
\begin{proof}
 If $z \in P(a)$ there exists a sequence $(z_{n_k})_{k\in \N}$ of points s. t. $z_{n_k}\in C_n(a)$ for all $k$ and $\abs{z-z_{n_k}}\le s_n$. For $m$ s. t. $t_m<t\le t_{m+1}$
 $$
 \left\lvert B(z_{n_k},t)-B(z_{n_k},t_1)-\sqrt{4a\pi^2}(t-t_1)\right\rvert=\o{t}
 $$
 follows as in \eqref{eq:o(t)}. Since the Brownian motion is a.s.\ continuous taking the limit for $k\to+\infty$
 $$
 \abs{B(z,t)-B(z,t_1)-\sqrt{4a\pi^2}(t-t_1)}=\o{t}
 $$
 and dividing by $\sqrt{2\pi^2}t$
 $$
\left|\frac{\ca{z}{G^{-1}(t)}}{\sqrt{2\pi^2}t}-\sqrt{2 a}\right|=\o{1}
 $$
 which is an equivalent formulation of the set of thick points.
 \end{proof}

Next we make preparations to define a measure $\mu$ supported on $P(a)$ with positive probability. For this purpose define a sequence of measures $\mu_n$ on $J$ supported on $n$-perfect thick points.
\begin{equation}
 \mu_n(\cdot)=\sum_{i=1}^{|C_n|}\frac{1}{\prob{E^n(x_{ni})}}\one_{\left\{Y_{ni}=1\right\}} \lambda\left(\cdot \cap S(x_{ni},s_n)\right),
\end{equation}
where $\lambda(\cdot)$ is the Lebesgue measure.

In the following lemma we list down some important properties of this measure.

\begin{lemma}\label{lemma:three_steps}
 Let $\mu_n(\cdot) $ be as above. Then the following hold:
\begin{itemize}
\item[(a)] $\E{\mu_n(J)}=1$;
\item[(b)] $\sup_{n}\E{\mu_n(J)^2}<\infty$;
\item[(c)] $\sup_n \E{I_\alpha(\mu_n)}<\infty$;
\item[(d)] there exist $a,b\in(0,\infty)$ such that for all $n$ we have
$$\prob{b\le \mu_n(J)<b^{-1}, I_\alpha(\mu_n)<a}>0$$
for any $\alpha \leq 4-a$.
\end{itemize}

\end{lemma}
The proof of Lemma \ref{lemma:three_steps} requires a correlation inequality and a lower bound depends on the following lemma. Its proof is similar to the proof of Lemma 3.3 of \cite{HMP} and hence we skip it.

\begin{lemma}\label{lemma:disjoint}
 Let $A_m(x), B_m(x)$ be as above with
 $s_m=\frac1{m!}$. Let $$E^n(x)= \bigcap_{m\le n} A_m(x) \cap B_{n+1}(x).$$

 Then
for every $y\in S(x,s_l)\setminus S(x, s_{l+1})$, $l>2 $, we have
\begin{equation}
 \prob{E^n(x)\cap E^n(y)}\le \mathscr C_l \prob{E^n(x)}\prob{E^n(y)},
 \end{equation}
 where $\mathscr C_l$ is defined by
 $$
\mathscr C_l:=C\prod_{j\le l+1} \f{1}{c_j},
$$
and $c_j= \exp\left(\f12 \sqrt{4a\pi^2}\sqrt{t_{j+1}-t_j}-4a\pi^2(t_{j+1}-t_j)\right)$.
\end{lemma}

Using the above Lemma the proof of Lemma~\ref{lemma:three_steps} is almost immediate.

\begin{proof}[Proof of Lemma~\ref{lemma:three_steps}]
Note the series $\sum_{l=1}^{\infty}s_l^4 \mathscr C_l$ converges (absolutely) by the ratio test. By means of the same criterion one shows also that $\sum_{l=1}^{\infty}s_l^4 \mathscr C_l s_{l+1}^{-\alpha}<+\infty$ under the assumption $\alpha\le 4$. Keeping these facts in mind we proceed to the proof.

\begin{itemize}
\item[(a)] As $S(x_{ni}, s_n)$ forms a cover of $J$ it is easy to show that $\E{\mu_n(J)}=1$.
In particular, \begin{align*}
  \E{\mu_n(J)}&=\sum_{i=1}^{|C_n|} \frac1{\prob{E^n(x_{ni})}}\prob{Y_{ni}=1} \lambda(J\cap S(x_{ni},s_n))\\
 &= \sum_{i=1}^{|C_n|}\lambda(J\cap S(x_{ni},s_n))=1.
 \end{align*}

\item[(b)] Using Lemma~\ref{lemma:disjoint} we have,
\begin{align*}
 \E{\mu_n(J)^2}&=\sum_{i,j=1}^{|C_n|}\frac{\prob{Y_{ni}=1, Y_{nj}=1}}{\prob{E^n(x_{ni})}\prob{E^n(x_{nj})}}\lambda(S(x_{ni},s_n))\lambda(S(x_{nj},s_n))\\
&\le s_n^8 \sum_{i=1}^{|C_n|}\sum_{l=1}^n\sum_{j=1,s_{l}\ge |x_{nj}-x_{ni}|>s_{l+1}}^{|C_n|} \frac{\prob{E^n(x_{ni})\cap E^n(x_{nj})}}{\prob{E^n(x_{ni})}\prob{E^n(x_{nj})}}\\
&\le s_n^8\sum_{i=1}^{|C_n|}\sum_{l=1}^n \frac{s_l^4}{s_n^4}\mathscr C_l \le \sum_{l=1}^{\infty}s_l^4 \mathscr C_l<\infty.
\end{align*}
Above we have used the fact that the number of hypercubes with center at $x_{ni}$ and radius $s_l$ is proportional to $s_l^4/s_n^4$.

\item[(c)] For the expected energy we follow the same procedure as above. Note that $|x_{ni}-x_{nj}|>s_{l+1}$ then if we take
$x\in S(x_{ni},s_n)$ and $y\in S(x_{nj},s_n)$ then $|x-y|>s_{l+1}$.
\begin{align*}
 \E{I_{\alpha}(\mu_n)}&=\sum_{i,j=1}^{|C_n|}  \frac{\prob{E^n(x_{ni})\cap E^n(x_{nj})}}{\prob{E^n(x_{ni})}\prob{E^n(x_{nj})}}\int_{S(x_{ni},s_n)}\int_{S(x_{nj},s_n)}\frac{\mathrm d x\mathrm d y}{|x-y|^{\alpha}}\\
&\le s_n^8\sum_{i=1}^{|C_n|}\sum_{l=1}^n \frac{s_l^4}{s_n^4}\mathscr C_l s_{l+1}^{-\alpha}\le \sum_{l\ge 1}\mathscr C_l s_l^4s_{l+1}^{-\alpha}<+\infty.
\end{align*}
\item[(d)]
By Paley-Zygmund inequality and the fact that $\sup_{n\ge 2} \E{\mu_n(J)^2}<\infty$, there exists $v>0$
\[
\prob{ \mu_n(J)\ge b} \ge (1-b)^2\f{1}{\E{\mu_n(J)}}\ge \frac{(1-b)^2}{\sup_{n\ge 2}\E{\mu_n(J)^2}}\ge v>0.
\]
Also using Markov's inequality we have that $$\prob{\mu_n(J)\ge b^{-1}}\le b \E{\mu_n(J)}=b.$$ Hence choosing $b>0$ and $v>0$ appropriately we have,
\[
\prob{b\le \mu_n(J)\le b^{-1}}= \prob{\mu_n(J)\ge b}- \prob{\mu_n(J)\ge b^{-1}}\ge 2v>0.
\]
 Also note that since $\E{I_\alpha(\mu_n)}$ is uniformly bounded in $n$, using Markov's inequality we have
$$\prob{I_\alpha(\mu_n)> a}\le v.$$
Hence (d) follows from the above observations and the fact that,
\begin{align*}
\prob{b\le \mu_n(J)\le b^{-1}, I_\alpha(\mu_n)\le a}&\ge \prob{b\le \mu_n(J)\le b^{-1}}-\prob{I_\alpha(\mu_n)> a}\\
&\ge 2v-v=v>0.
\end{align*}
\end{itemize}
\end{proof}
\begin{proof}[Proof of the lower bound] Now using Lemma~\ref{lemma:three_steps} we continue with the proof of lower bound.
If we define
$$
G:=\limsup_{n\to+\infty}\left\{b\le \mu_n(J)<b^{-1}, I_\alpha(\mu_n)<a \right\},
$$
then by Lemma \ref{lemma:three_steps} (d), $\prob{G}$ is bounded away from zero. $I_\alpha$ being a lower semicontinuous function, the set of measures $\mu$ for which
$b\le \mu(J)<b^{-1}$ and $I_\alpha(\mu)<a$ is compact in the topology of weak convergence. Therefore the sequence
$(\mu_n)_{n\in \N}$ admits surely along a subsequence $(\mu_{n_k})_{k\in \N}$ a weak limit $\mu$, which is a finite measure supported on $P(a)$ and whose
$\alpha$-energy is finite. Hence, we have
\begin{equation}\label{lower:pos}
\prob{C_{\mathcal H}^{4-a}(P(a))>0}>0.
\end{equation}
Now by the monotonicity of the Hausdorf\/f-$\alpha$-measure, if we can show that $$\prob{C_{\mathcal H}^{4-a}(T(a,J))>0}\in\left\{0,1\right\}$$
then by~\eqref{lower:pos},  the set
$\left\{C_{\mathcal H}^{4-a}(T(a,J))>0\right\}$ will have probability one and hence the proof will be complete.

Now from the construction of $\mu^x_\epsilon$, it holds from Equation~7.9 of \cite{CJ} that $\ca{x}{\epsilon}=
f_1(\epsilon)\mathcal I(h_{\sigma^x_\epsilon})+f_2(\epsilon)\mathcal I(h_{\mathrm d \sigma^x_\epsilon}) $, where
$$
f_1(\epsilon)=\f{\epsilon I_1(\epsilon)-2 I_2(\epsilon)}{I_1^2(\epsilon)-I_0(\epsilon)I_2(\epsilon) },\quad f_2(\epsilon)=\f{-\epsilon  I_2(\epsilon)}{I_1^2(\epsilon)-I_0(\epsilon)I_2(\epsilon) }.
$$ Since $\lim_{\eps\to 0}f_1(\epsilon)=2$ and $\lim_{\eps\to 0}f_2(\epsilon)=0$, $\mu^x_\epsilon\to 2\delta_x$ as $\eps\to 0$ in the sense of
distributions. In fact, since $\widehat{\mathrm d \sigma^x_\eps}(\xi)=-\f{2}{\eps}J_2(\eps|\xi|)\exp\left(i (\xi,x)_{\R^4}\right)\to 0$ for all $\xi $, $\mathrm d \sigma^x_\eps\to 0$ in the sense of distributions.
Thus
\begin{eqnarray*}
&&\limsup_{\eps \to 0} \f{\ca{x}{\eps}}{\sqrt{2\pi}G(\eps)}=\limsup_{\eps \to 0}\f{f_1(\epsilon)\mathcal I(h_{\sigma^x_\epsilon})}{\sqrt{2\pi}G(\eps)}.
\end{eqnarray*}
By \cite{Str08} (Section 2),  if $\{h_m\}_{m\in\N}$ is an orthonormal basis of $H$,
$$[\mathcal I(h_{\sigma^x_\epsilon})](\theta)=\la \theta,\,\sigma^x_\eps\ra\stackrel{\mathcal W-\mbox{a.s.\ }}{=}\la\sum_{m\geq 1}[\mathcal I(h_m)(\theta)]h_m,\,\sigma^x_\eps \ra.$$
The series will depend then only on its tail, as $\la h_m,\,\sigma^x_\eps\ra\to h_m(x)$ and $G(\eps)\to+\infty$ as $\eps \to 0$. Using the fact that $(\mathcal I(h_m))_{m\geq 1}$ are i.i.d.\ we can apply Kolmogorov's 0-1 law to conclude.\end{proof}

\section{Appendix}
 Here we will collect some of the bounds on the Bessel functions. These bounds are easy to derive but for completeness we provide a short proof for them.

\begin{lemma}\label{lemma:bound_G}
(a) For some constant $C>0$ and $x>0$
 $$\abs{I^2_1(x)-I_0(x)I_2(x) }\ge C x^2.$$
(b) Let $G(\cdot)$ be as in~\eqref{eq:green}, then $G(x)\leq -C \log x$ for all $x\in [0,1]$, with $C>0$ uniform in $x$.
\end{lemma}
\begin{proof}
(a) Following \cite{JosBis} we have,
\begin{align*}
 I^2_1(x)-I_0(x)I_2(x) &=\frac{I_1^2(x)}{x}\left(x\frac{I_1'(x)}{I_1(x)}\right)'=\frac{I_1^2(x)}{x}\sum_{n\ge 1}\frac{4 x j_{1,n}}{(x^2+j_{1,n}^2)^2}.
\end{align*}
where we used the equality $\left(x\frac{I_1'(x)}{I_1(x)}\right)'=\sum_{n\ge 1}\frac{4 x j_{1,n}}{(x^2+j_{1,n}^2)^2}$, $j_{i,n}$ being
the $n$-th zero of $J_1(x)/x$ (\cite{Wat44}). Now using the identity $I_1(x)=(x/C)\prod_{n \ge 1}\left(1+\frac{x^2}{j_{1,n}^2}\right)$ (\citep[Page 498]{Wat44}) we derive
\begin{align*}
I^2_1(x)-I_0(x)I_2(x)&=\frac{I_1^2(x)}{x}\left(x\frac{I_1'(x)}{I_1(x)}\right)'\\
&=\frac{I_1^2(x)}{x}\frac{4 x j_{1,1}}{(x^2+j_{1,1}^2)^2}+\frac{I_1^2(x)}{x}\sum_{n\ge 2}\frac{4 x j_{1,n}}{(x^2+j_{1,n}^2)^2}\\
&>\frac{4 I_1^2(x) j_{1,1}}{(x^2+j_{1,1}^2)^2}>C'x^2.
\end{align*}
(b) By part (a) and the series expansion of Bessel functions (\cite{AbrSte}) one can find a bound for $G(\cdot)$ as follows ($\gamma$ is the Euler-Mascheroni constant):
\begin{small}
\begin{align*}
 G(x)&\leq \frac{C}{x^2}(2I_1(x)K_1(x)+2I_2(x)K_0(x)-1)\\
&=\frac{C}{x^2}\left(2\left(\frac{x}{2}+\frac{x^3}{16}+\O{x^4}\right)\left(\frac{1}{x}+\frac{x}{4}\left(-1+2\gamma-2\log 2+2\log x\right)+\O{x^3\log x}\right)\right.\\
&+\left.2\left(\frac{x^2}{8}+\O{x^4}\right)\left((-\gamma+\log 2-\log x)+\right.\right.+\left.\left.\O{x^2\log x}\right)-1\right)\\
&=\frac{C}{x^2}\left(1+\frac{x^2}{8}+\O{x^3}\right.+\left.\frac{-1+2\gamma-2\log 2}{4}x^2+\frac{-1+2\gamma-2\log 2}{32}x^4+\right.\\
&+\O{x^2\log x}+\left. \frac{x^2}{4}C+\O{x^4}-\frac{x^2\log x}{4}-1\right)=-C\log x+C'.
\end{align*}
\end{small}
Here $C,\,C'$ denote positive constants that may vary from line to line.
\end{proof}

 \section{Acknowledgements}
We thank Erwin Bolthausen, Linan Chen and Jason Miller for some helpful discussions.
\bibliography{literatur}

\begin{thebibliography}{}

\bibitem[Abramowitz and Stegun, 1965]{AbrSte}
Abramowitz, M. and Stegun, I.~A. (1965).
\newblock {\em Handbook of Mathematical Functions: with Formulas, Graphs, and
  Mathematical Tables (Dover Books on Mathematics)}.
\newblock Dover books on mathematics. Dover Publications, 1 edition.

\bibitem[{Chen} and {Jakobson}, 2012]{CJ}
{Chen}, L. and {Jakobson}, D. (2012).
\newblock {Gaussian Free Fields and KPZ Relation in R\^{}4}.
\newblock {\em To appear in Annales Henri Poincar\'e. Preprint available in
  arXiv:1210.8051.}

\bibitem[{Cipriani}, 2013]{Cip13}
{Cipriani}, A. (2013).
\newblock {High points for the membrane model in the critical dimension}.
\newblock {\em ArXiv e-prints}.

\bibitem[Daviaud, 2006]{Daviaud}
Daviaud, O. (2006).
\newblock {Extremes of the discrete two-dimensional Gaussian Free Field}.
\newblock {\em The Annals of Probability}, 34(3):962--986.

\bibitem[Dembo et~al., 2000]{DPRZSpatial}
Dembo, A., Peres, Y., Rosen, J., and Zeitouni, O. (2000).
\newblock {Thick points for spatial Brownian motion: multifractal analysis of
  occupation measure.}
\newblock {\em Ann. Probab.}, 28(1):1--35.

\bibitem[Dembo et~al., 2001]{DPRZPlanar}
Dembo, A., Peres, Y., Rosen, J., and Zeitouni, O. (2001).
\newblock {Thick points for planar Brownian motion and the Erd\H{o}s-Taylor
  conjecture on random walk.}
\newblock {\em Acta Math.}, 186(2):239--270.

\bibitem[Duplantier and Sheffield, 2011]{DS10}
Duplantier, B. and Sheffield, S. (2011).
\newblock {Liouville quantum gravity and KPZ}.
\newblock {\em Inventiones mathematicae}, 185(2):333--393.

\bibitem[Hu et~al., 2010]{HMP}
Hu, X., Miller, J., and Peres, Y. (2010).
\newblock {Thick points of the Gaussian free field.}
\newblock {\em Ann. Probab.}, 38(2):896--926.

\bibitem[Joshi and Bissu, 1991]{JosBis}
Joshi, C. and Bissu, S. (1991).
\newblock {Some inequalities of Bessel and modified Bessel functions.}
\newblock {\em J. Aust. Math. Soc., Ser. A}, 50(2):333--342.

\bibitem[Kahane, 1985]{Kah85}
Kahane, J.-P. (1985).
\newblock Sur le chaos multiplicatif.
\newblock {\em Ann. Sci. Math. Qu\'ebec}, 9(2):105--150.

\bibitem[Knizhnik et~al., 1988]{KPZ}
Knizhnik, V.~G., Polyakov, A.~M., and Zamolodchikov, A.~B. (1988).
\newblock Fractal structure of {$2$}{D}-quantum gravity.
\newblock {\em Modern Phys. Lett. A}, 3(8):819--826.

\bibitem[Kurt, 2007]{Kurt_d5}
Kurt, N. (2007).
\newblock {Entropic repulsion for a class of Gaussian interface models in high
  dimensions.}
\newblock {\em Stochastic Processes Appl.}, 117(1):23--34.

\bibitem[Kurt, 2008]{Kurt_thesis}
Kurt, N. (2008).
\newblock {\em {Entropic repulsion for a Gaussian membrane model in the
  critical and supercritical dimension}}.
\newblock PhD thesis, University of Zurich.

\bibitem[Kurt, 2009]{Kurt_d4}
Kurt, N. (2009).
\newblock {Maximum and entropic repulsion for a Gaussian membrane model in the
  critical dimension}.
\newblock {\em The Annals of Probability}, 37(2):687--725.

\bibitem[Mandelbrot, 1972]{Man72}
Mandelbrot, B. (1972).
\newblock Possible refinement of the lognormal hypothesis concerning the
  distribution of energy dissipation in intermittent turbulence.
\newblock In Rosenblatt, M. and Atta, C., editors, {\em Statistical Models and
  Turbulence}, volume~12 of {\em Lecture Notes in Physics}, pages 333--351.
  Springer Berlin Heidelberg.

\bibitem[Mandelbrot et~al., 2004]{BarMan04}
Mandelbrot, B., Lapidus, M., and Van~Frankenhuysen, M. (2004).
\newblock {\em Fractal Geometry and Applications: Multifractals, probability
  and statistical mechanics, applications}.
\newblock Fractal Geometry and Applications: A Jubilee of Beno{\^\i}t
  Mandelbrot : Analysis, Number Theory, and Dynamical Systems. American
  Mathematical Society.

\bibitem[M{\"o}rters et~al., 2010]{MoerPer}
M{\"o}rters, P., Peres, Y., Schramm, O., and Werner, W. (2010).
\newblock {\em Brownian Motion}.
\newblock Cambridge Series in Statistical and Probabilistic Mathematics.
  Cambridge University Press.

\bibitem[Rhodes and Vargas, 2013]{RhoVarRev}
Rhodes, R. and Vargas, V. (2013).
\newblock Gaussian multiplicative chaos and applications: a review.

\bibitem[Robert and Vargas, 2008]{RobVar08}
Robert, R. and Vargas, V. (2008).
\newblock Hydrodynamic turbulence and intermittent random fields.
\newblock {\em Communications in Mathematical Physics}, 284(3):649--673.

\bibitem[{Robert} and {Vargas}, 2010]{RaoVar10}
{Robert}, R. and {Vargas}, V. (2010).
\newblock {Gaussian multiplicative chaos revisited.}
\newblock {\em {Ann. Probab.}}, 38(2):605--631.

\bibitem[Stroock, 2008]{Str08}
Stroock, D. (2008).
\newblock Abstract wiener space, revisited.
\newblock {\em Comm. Stoch. Anal.}, 2(1):145--151.

\bibitem[Stroock, 2010]{Str10}
Stroock, D. (2010).
\newblock {\em Probability Theory: An Analytic View}.
\newblock Cambridge University Press.

\bibitem[Watson, 1944]{Wat44}
Watson, G.~N. (1944).
\newblock {\em A treatise on the theory of Bessel functions}.
\newblock Cambridge University Press.

\end{thebibliography}
\bibliographystyle{apalike}

\end{document}